\documentclass[11pt,leqno]{amsart}

\usepackage[utf8x]{inputenc}
\usepackage[T1]{fontenc}
\usepackage[english]{babel}

\usepackage{enumitem}

\usepackage{amsmath,amssymb,amsthm}
\usepackage{amsrefs}
\usepackage[dvipsnames]{xcolor}
\usepackage[colorlinks=true, linkcolor=MidnightBlue, citecolor=OliveGreen]{hyperref}
\usepackage{mathdots}

\usepackage{tikz}

\theoremstyle{definition}
\newtheorem{definition}{{\bf Definition}}[section]
\theoremstyle{theorem}
\newtheorem{lemma}[definition]{{\bf Lemma}}
\newtheorem{theorem}{{\bf Theorem}}
\newtheorem{cor}[definition]{{\bf Corollary}}
\newtheorem{prop}[definition]{{\bf Proposition}}

\DeclareMathOperator{\Z}{\mathbb{Z}}
\DeclareMathOperator{\N}{\mathbb{N}}
\DeclareMathOperator{\id}{id}

\DeclareMathOperator{\lcm}{lcm}
\DeclareMathOperator{\ord}{ord}
\DeclareMathOperator{\Aut}{Aut}
\DeclareMathOperator{\st}{st}
\DeclareMathOperator{\St}{St}
\DeclareMathOperator{\rst}{rist}
\DeclareMathOperator{\Rst}{Rist}
\newcommand{\cref}[3][]{\hyperref[#3]{#2~\ref*{#3}#1}}
\DeclareMathOperator{\tetr}{tetr}
\DeclareMathOperator{\slog}{slog}

\makeatletter
\@namedef{subjclassname@2020}{\textup{2020} Mathematics Subject Classification}
\makeatother

\title{Groups of small period growth}

\author[J. M. Petschick]{Jan Moritz Petschick}
\address{Jan Moritz Petschick: Mathematisches Institut, Heinrich-Heine-Universit\"at, 40225 D\"usseldorf, Germany}
\email{jan.petschick@hhu.de}

\thanks{The research was funded by the Deutsche Forschungsgemeinschaft (DFG, German Research Foundation) — 380258175}
\date{\today}

\begin{document}

\begin{abstract}
	We construct finitely generated groups of small period growth, i.e.\ groups where the maximum order of an element of word length $n$ grows very slowly in $n$. This answers a question of Bradford related to the lawlessness growth of groups and is connected to an approximative version of the restricted Burnside problem.
\end{abstract}

\maketitle

\section{Introduction} 
\label{sec:introduction}

In this paper we provide an affirmative answer to the following question posed by Henry~Bradford at the ``New Trends around Profinite Groups'' conference in Levico~Terme, 2021.
\begin{description}
	\item[Q1] Is there a lawless finitely generated $p$-group of sublinear period growth?
\end{description}
Let $G$ be a group generated by a finite set $S$. For any $n \in \N$ write $B_G^S(n)$ for the set of elements in $G$ of word length at most $n$ (with respect to $S$). The \emph{period growth function $\pi_G^S: \N \to \N \cup \{\infty\}$ of $G$ with respect to $S$}, first considered by Grigorchuk \cite{Gri83}, is defined by
\[
	\pi_G^S(n) = \max \{\ord(g) \mid g \in B_G^S(n) \}.
\]
Grigorchuk proved that the {growth type} of $\pi_G^S$ is independent of the choice of $S$. Consequently, \textbf{Q1} is well-posed and we drop the superscript $S$ in statements regarding the growth type of the period growth function of a group.

Bradford’s question was motivated by an application to \emph{lawlessness growth}, cf.\ \cite{Bra}*{Example 2.7 \& Question 10.2}. The lawlessness growth of a lawless group measures the minimal word length of witnesses to the non-triviality of the verbal subgroup $w(G)$ for group words $w$ of increasing length.
Since elements of order $m$ do not satisfy any power words of length smaller than $m$, there is a connexion to the period growth of $G$. In fact, an example of a lawless $p$-group, $p$ being some prime, with the properties required by \textbf{Q1} has super-linear lawlessness growth. For a detailed study on lawlessness growth, we refer to \cite{Bra}.

Clearly a group with the properties demanded in \textbf{Q1} is infinite, since it is lawless, and periodic, since otherwise there exists some $n_0 \in \N$ such that $\pi_G^S(n) = \infty$ for all $n \geq n_0$. Little is known regarding the period growth of finitely generated infinite periodic groups. Grigorchuk proved that the (first) Grigorchuk group $\mathcal G$ fulfills $\pi_{\mathcal G} \precsim n^9$, where, given two non-decreasing functions $f, g: \N \to \mathbb R_{>0}$, we write $f \precsim g$ if $\limsup\limits_{n \to \infty}{f(n)}/{g(n)} < \infty$. This bound was improved by Bartholdi and \v{S}uni\'{k} \cite{BS01} to $n^{3/2}$, also extending the result to certain generalisations of $\mathcal G$. In \cite{Bra}*{Remark 5.7} Bradford constructs a Golod--Shafarevich $p$-group of at most linear period growth. We remark that the standard proof that the Gupta--Sidki $3$-group $\Gamma_3$ is periodic yields $\pi_{\Gamma_3} \precsim n^{1/\log_3(4/3)}$.

To state our main result, we need to define some functions growing very slowly. The \emph{tetration function $\tetr_{k}: \N \to \N$ with base $k$} is defined recursively by $\tetr_{k}(0) = 1$ and $\tetr_{k}(n+1) = k^{\tetr_{k}(n)}$ for $n \in \N$. We define a left-inverse non-decreasing function by $\slog_{k}(n) = \max \{ l \in \N \mid \tetr_{k}(l) \leq n \}$.

Now we may state our main result.
\begin{theorem}\label{thm:main}
	There exists a $4$-generated infinite residually finite periodic $2$-group $G$ such that
	\[
		\pi_G \precsim \exp_8 \circ \slog_{2}.
	\]
\end{theorem}
In particular, the function $\pi_G$ grows slower than any iterated logarithm. \cref{Theorem}{thm:main} gives an affirmative answer to \textbf{Q1}.

The group we construct to prove \cref{Theorem}{thm:main} is realised as a group of automorphisms of a spherically homogeneous locally finite rooted tree, whose valency is unbounded. In the theory of automorphisms of rooted trees it is often interesting to obtain examples acting on regular trees, i.e.\ locally finite trees where all vertices (except the root vertex) have the same valency. On our way to prove \cref{Theorem}{thm:main}, we obtain a family of groups of slow (albeit far faster than the growth described in \cref{Theorem}{thm:main}) period growth that act on regular rooted trees without additional work.

\begin{theorem}\label{thm:secondary}
	Let $\epsilon > 0$. There exists a finitely generated infinite residually finite periodic $2$-group $G_\epsilon$ acting on a regular rooted tree (depending on $\epsilon$) such that
	\[
		\pi_{G_\epsilon} \precsim n^\epsilon.
	\]
\end{theorem}

We stress the fact that the groups we construct are residually finite. This is important in the context of the following approximative variant of the restricted Burnside problem. The restricted Burnside problem may be formulated as: Are residually finite groups with bounded period growth function finite? Thus, considering groups with slow but not bounded period growth as the next best thing to groups of finite exponent, we ask:
\begin{description}
	\item[Q2] Among all $m$-generated residually finite infinite $p$-groups $G$, what are the minimal growth types of $\pi_G$?
\end{description}
By Zel\cprime manovs \cites{Zel90, Zel91} solution to the restricted Burnside problem, the finite residual $\mathrm{res~B}(m,n)$ of the free Burnside group of rank $m$ and exponent $n$ is a finite group for all values of $m$ and $n$. Define
\[
	\operatorname{zel}_m(n) = \max \{ k \in \N \mid |\mathrm{res~B}(m,k)| \leq n\}.
\]
Since \textbf{Q2} excludes finite groups, this function yields a lower bound for the period growth function of any $m$-generated residually finite infinite $p$-group. The best known lower bound for $\operatorname{zel}_m(n)$ is due to Groves and Vaughan-Lee \cite{GVL03}, who prove that
\[
	\operatorname{zel}_m(n^{(4^n)}) \geq \slog_m(n). 
\]
\cref{Theorem}{thm:main} provides a group whose period growth comes close to the best known upper bound for $\operatorname{zel}_m$,
\[
	\operatorname{zel}_m(2^{2^{\iddots^{2^m}}}) \leq 2^k,
\]
with $k$ appearences of the number $2$ in the tower on the left side, which is due to Newman, whose argument is given in \cite{VLZ99}.

\subsection*{Organisation}
After some preliminary definitions, we first prove \cref{Theorem}{thm:secondary}, and then use the groups constructed for this purpose as a model for the more involved construction of the group we use to prove \cref{Theorem}{thm:main}. We then establish that all the groups constructed are lawless and thus constitute examples of groups with fast lawlessness growth. We end with some open questions related to the subject.

\subsection*{Acknowledgements}
This is part of the author's Ph.D. thesis, written under the supervision of Benjamin Klopsch at the Heinrich-Heine-Universität Düsseldorf. The author thanks Henry Bradford for his inspiring questions and for granting the author access to his unpublished work.


\section{Groups of automorphisms of rooted trees} 
\label{sec:preliminaries}

Let $G$ be a group generated by a set $S$. We write $||\cdot||_S: G \to \N$ for the word length function of $G$ with respect to $S$ and $B_G^S(n)$ for the set of elements of $G$ of length $n$ with respect to $S$. For two integers $l, u \in \Z$, we denote by $[l, u]$ and $[l, u)$ the set of integer numbers within the corresponding intervals.

Let $(X_n)_{n \in \N_+}$ be a sequence of finite non-empty sets. The \emph{(spherically homogeneous) rooted tree of type} $(X_n)_{n \in \N_+}$ is the tree $T$ with finite strings $x_1 \dots x_k$, $x_i \in X_i$ for $i \in [1,k]$, as vertices and edges between strings that only differ by one letter. The empty string is called the root of the tree. Every vertex of distance $k$ for some fixed $k \in \N$ from the root is a string of length $k$, which has valency $|X_{k+1}|+1$. The set $\mathcal{L}_T(k)$ of vertices of vertices of distance $k$ to the root is called the $k$\textsuperscript{th} layer of the tree. We identify the first layer with the set $X_1$. Every vertex $u \in \mathcal{L}_T(k)$ is the root of a rooted subtree $T_u$ of type $(X_n)_{n \geq k}$. We may compose strings in the following way: if $v \in \mathcal{L}_T(k)$ and $u \in T_v$, then the concatenation $vu$ is a vertex of $T$.

If the sequence $(X_n)_{n \in \N_+}$ is constant, we call the corresponding tree \emph{regular}. In this case, all subtrees $T_u$ for $u \in T$ are isomorphic.

A \emph{(tree) automorphism} of $T$ is a (graph) automorphism of $T$ fixing the root. Such a map must also leave the layers of $T$ invariant. Let $v \in T$ and $u \in T_v$ be two vertices, and $a \in \Aut(T)$ an automorphism of $T$. Then the equation
\[
	(vu).a = (v.a) (u.(a|_v))
\]
defines a unique automorphism $a|_v$ of $T_v$ called the \emph{section of $a$ at $v$}.

Any automorphism $a$ can be decomposed into its sections prescribing the action at the subtrees of the first layer, and $a|^\epsilon$, the action of $a$ on the first layer $\mathcal{L}_T(1) = X_1$. We adopt the convention that an $X_1$-indexed family $(x: a_x)_{x \in X_1}$ of automorphisms $a_x \in \Aut(T_x)$ is identified with the automorphism having section $a_x$ at $x$ which stabilises the first layer. Hence for any $a \in \Aut(T)$ we write
\[
	a = (x: a|_x)_{x \in X_1} a|^\epsilon.
\]
We record some important equalities for sections. Let $a \in \Aut(T), u \in T$ and $v \in T_u$. Then
\[\begin{array}{ccc}
	(a|_u)|_v = a|_{uv}, &
	(ab)|_u = a|_ub|_{u.a}, &
	a^{-1}|_u = (a|_{u.a^{-1}})^{-1}.
\end{array}\]
We call an automorphism \emph{rooted} if all its first layer sections are trivial, i.e.\ if it permutes the set of subtrees $\{T_x \mid x \in X_1 \}$. The subgroup of rooted automorphisms is isomorphic to $\operatorname{Sym}(X_1)$.

Let $G \leq \Aut(T)$ be a group of automorphisms. The (pointwise) stabiliser of the $k$\textsuperscript{th} layer of $T$ in $G$ is denoted $\St_G(k)$ and called the \emph{$k$\textsuperscript{th} layer stabiliser}. All layer stabilisers are normal subgroups of finite index in $G$. Their intersection is trivial, hence the group $G$ is residually finite. The group $G$ is called \emph{spherically transitive} if it acts transitively on every layer $\mathcal{L}_T(k)$. 

The \emph{$k$\textsuperscript{th} rigid layer stabiliser} $\Rst_G(k)$ of a spherically transitive group $G$ for some $k \in \N$ is the product of all (equivalently, the normal closure of a) \emph{rigid vertex stabiliser} $\rst_G(u) = \{g \in G \mid g|_v = \id \text{ for } v \in T\setminus T_u\}$, where $u \in \mathcal{L}_T(k)$. A spherically transitive group $G$ is \emph{weakly branch} if $\Rst_G(k)$ is non-trivial for all $k \in \N$. Every weakly branch group is lawless, cf.\ \cite{Abe05}.

If $T$ is regular, a group $G \leq \Aut(T)$ is called \emph{self-similar} if for all $u \in T$ the image of the section map $G|_u$ is contained in $G$. It is called \emph{fractal} if $\st_G(x)|_x = G$ for all $x \in \mathcal{L}_T(1)$. The group $G$ is called \emph{weakly regular branch} if it contains a non-trivial subgroup $H \leq G$ such that $\rst_H(x) \geq H$ for all $x \in \mathcal{L}_T(1)$. Every weakly regular branch group is weakly branch.

Since we aim to provide examples of periodic groups, we need the following criterion for periodicity, which is adopted from the methods developed by Grigorchuk, Gupta and Sidki, cf.\ \cites{Gri83,GS83}. Since our criterion is adapted to a more general situation, we give a short proof.

\begin{prop}\label{prop:periodicity}
	Let $G \leq \Aut(T)$ be a group, let $\pi$ be a set of primes and let $n \in \N$ be a positive integer, such that $G|_u/\St_{G|_u}(n)$ is a $\pi$-group for $u \in T$. For every vertex $u \in T$, let $||\cdot ||_u: G|_u \to \N$ be a length function such that $||g||_u \leq 1$ implies that $g$ is a $\pi$-element.
	
	If for all vertices $u, v \in T$ such that $v = uw$ for some string $w$ of length $n$, and all $g \in G|_u$ we have
	\begin{equation*}\label{eq:length reduction}
		||g|_w||_v < ||g||_u/\exp(G|_u/{\St_{G|_u}(n)}),\tag{$\star$}
	\end{equation*}
	then $G$ is a $\pi$-group.
\end{prop}

\begin{proof}
	Let $g \in G|_u$ for some $u \in \mathcal{L}_T(k)$ and $k \in \N$. We prove that the order of $g$ is finite and divisible by primes in $\pi$ only. The statement then is obtained by considering $u = \epsilon$. We use induction on $\ell = ||g||_{u}$. If $\ell \leq 1$, the element is a $\pi$-element by assumption. If $\ell > 1$, write $q = \exp(G|_u/\St_{G|_u}(n))$. By assumption, $q$ is only divisible by primes in $\pi$. Now $g^q$ stabilises the $n$\textsuperscript{th} layer, hence $g^q = (x: g^q|_x)_{x \in \mathcal{L}_{T_u}(n)}$ and $\ord(g) | q \cdot \lcm\{ \ord(g^q|_x) \mid x \in \mathcal{L}_{T_u}(n) \}$. Using (\ref{eq:length reduction}) we obtain
	\[
		||g^q|_x||_{ux} < ||g^q||_u/q \leq ||g||_u = \ell
	\]
	for all $x \in \mathcal{L}_{T_u}(n)$. Thus by induction $\ord(g^q|_x)$ finite and divisible by primes in $\pi$ only, and consequentely the same holds for $g$.
\end{proof}


\section{Layerwise length reduction and the proof of \cref{Theorem}{thm:secondary}} 
\label{sec:proof_of_the_theorem}

We construct a family of groups $K_r$, indexed by the positive integers, acting on regular rooted trees $T^{(r)}$ whose type depends on $r$. Fix a positive integer $r$, and write $A_r = C_2^r$ for the elementary abelian $2$-group of rank $r$. Also fix a (minimal) generating set $E_r = \{e_i \mid i \in [0, r)\}$. Let $T^{(r)}$ be the regular rooted tree of type $(A_r)_{n \in \N_+}$. We now construct $K_r$ as a group of automorphisms of $T^{(r)}$, using a construction much in spirit of the Gupta--Sidki $p$-groups or the second Grigorchuk group. In fact, $K_r$ is a (constant) spinal group in the terminology of \cites{BGS03, Pet21}.

View the group $A_r$ as rooted automorphisms of $T^{(r)}$ by embedding $A_r$ into $\operatorname{Sym}(A_r)$ via its right multiplication action. Notice that we may see an element $a \in A_r$ both a as vertex of $T^{(r)}$ and an automorphism acting on $T^{(r)}$. We fix a translation map of $A_r$, given by $a \mapsto \overline{a} := \prod_{i = 0}^{r-1}e_i a$. Therefore $||\overline{e_i}||_{E_r} = r - 1$ for all $i \in [0,r)$.

Define $b_r \in \Aut(T^{(r)})$ by
\[
	b_r = (1_{A_r}: b_r;\,\, \overline{e_i}: e_i \text{ for } i \in [0, r);\,\, \ast: \id),
\]
where $\ast$ stands for every element of $A_r$ not refered to elsewhere in the tuple. \cref{Figure}{fig:generator} depicts the case $r = 3$ as an example. Notice that $b_r \in \St(1)$ is an involution. We define
\[
	K_r = \langle A_r \cup \{b_r\} \rangle.
\]
This is a group generated by $r+1$ involutions. For $r = 1$ we obtain a group isomorphic to the infinite dihedral group, and also $P_2$ contains elements of infinite order, but for $r > 2$ all groups $K_r$ are periodic by \cite[Theorem A]{Pet21}. We do not need to rely on this result, since the bounds establishing slow period growth also show that $K_r$ is periodic for $r > 4$. Since we are mostly interested in $K_r$ for big $r$, this suffices for our purposes.

\begin{figure}
	\begin{tikzpicture}
		\draw (0,0) -- (2,0) -- (2,2) -- (0,2) -- (0,0);
		\draw (0,2) -- (1,3) -- (3,3) -- (3,1) -- (2,0);
		\draw (2,2) -- (3,3);
		\draw[dotted] (0,0) -- (1,1) -- (1,3);
		\draw[dotted] (1,1) -- (3,1);
		
		\draw[red, dashed] (-.2,0) -- (-1.8,0);
		\node at (0,-.3) {$1_{A_3}$};
		\node at (-2.2,-.7) {$\curvearrowleft b_3$};
		
		\begin{scope}[shift={(-2,0)}, shift={(-0.6,-0.45)}, scale=0.3]
			\draw (0,0) -- (2,0) -- (2,2) -- (0,2) -- (0,0);
			\draw (0,2) -- (1,3) -- (3,3) -- (3,1) -- (2,0);
			\draw (2,2) -- (3,3);
			\draw[dotted] (0,0) -- (1,1) -- (1,3);
			\draw[dotted] (1,1) -- (3,1);
		\end{scope}
		
		\draw[red, dashed] (-.2,2) -- (-1.8,2);
		\node at (-.1,2.3) {$e_1$};
		
		\begin{scope}[shift={(-2,2)}, shift={(-0.6,-0.45)}, scale=0.3]
			\draw (0,0) -- (2,0) -- (2,2) -- (0,2) -- (0,0);
			\draw (0,2) -- (1,3) -- (3,3) -- (3,1) -- (2,0);
			\draw (2,2) -- (3,3);
			\draw[dotted] (0,0) -- (1,1) -- (1,3);
			\draw[dotted] (1,1) -- (3,1);
		\end{scope}
		
		\draw[red, dashed] (.8,1) -- (-.8,1);
		\node at (1.1,.7) {$e_2$};
		
		\begin{scope}[shift={(-1,1)}, shift={(-0.6,-0.45)}, scale=0.3]
			\draw (0,0) -- (2,0) -- (2,2) -- (0,2) -- (0,0);
			\draw (0,2) -- (1,3) -- (3,3) -- (3,1) -- (2,0);
			\draw (2,2) -- (3,3);
			\draw[dotted] (0,0) -- (1,1) -- (1,3);
			\draw[dotted] (1,1) -- (3,1);
		\end{scope}
		
		\draw[red, dashed] (.8,3) -- (-.8,3);
		\node at (1,3.3) {$\overline{e_0}$};
		
		\begin{scope}[shift={(-1,3)}, shift={(-0.6,-0.45)}, scale=0.3]
			\filldraw[gray!30] (1,0) -- (1,2) -- (2,3) -- (2,1);

			\draw (0,0) -- (2,0) -- (2,2) -- (0,2) -- (0,0);
			\draw (0,2) -- (1,3) -- (3,3) -- (3,1) -- (2,0);
			\draw (2,2) -- (3,3);
			\draw[dotted] (0,0) -- (1,1) -- (1,3);
			\draw[dotted] (1,1) -- (3,1);
		\end{scope}
		
		\draw[red, dashed] (2.2,0) -- (3.8,0);
		\node at (2,-.3) {$e_0$};
		
		\begin{scope}[shift={(4,0)}, shift={(-0.2,-0.45)}, scale=0.3]
			\draw (0,0) -- (2,0) -- (2,2) -- (0,2) -- (0,0);
			\draw (0,2) -- (1,3) -- (3,3) -- (3,1) -- (2,0);
			\draw (2,2) -- (3,3);
			\draw[dotted] (0,0) -- (1,1) -- (1,3);
			\draw[dotted] (1,1) -- (3,1);
		\end{scope}
		
		\draw[red, dashed] (2.2,2) -- (3.8,2);
		\node at (1.9,2.3) {$\overline{e_2}$};
		
		\begin{scope}[shift={(4,2)}, shift={(-0.2,-0.45)}, scale=0.3]
			\filldraw[gray!30] (.5,.5) -- (.5,2.5) -- (2.5,2.5) -- (2.5,.5);
			
			\draw (0,0) -- (2,0) -- (2,2) -- (0,2) -- (0,0);
			\draw (0,2) -- (1,3) -- (3,3) -- (3,1) -- (2,0);
			\draw (2,2) -- (3,3);
			\draw[dotted] (0,0) -- (1,1) -- (1,3);
			\draw[dotted] (1,1) -- (3,1);
		\end{scope}
		
		\draw[red, dashed] (3.2,1) -- (4.8,1);
		\node at (3.1,.7) {$\overline{e_1}$};
		
		\begin{scope}[shift={(5,1)}, shift={(-.2,-0.45)}, scale=0.3]
			\filldraw[gray!30] (0,1) -- (1,2) -- (3,2) -- (2,1);
		
			\draw (0,0) -- (2,0) -- (2,2) -- (0,2) -- (0,0);
			\draw (0,2) -- (1,3) -- (3,3) -- (3,1) -- (2,0);
			\draw (2,2) -- (3,3);
			\draw[dotted] (0,0) -- (1,1) -- (1,3);
			\draw[dotted] (1,1) -- (3,1);
		\end{scope}
		
		\draw[red, dashed] (3.2,3) -- (4.8,3);
		\node at (3,3.3) {${\overline{1_{A_3}}}$};
		
		\begin{scope}[shift={(5,3)}, shift={(-.2,-0.45)}, scale=0.3]
			\draw (0,0) -- (2,0) -- (2,2) -- (0,2) -- (0,0);
			\draw (0,2) -- (1,3) -- (3,3) -- (3,1) -- (2,0);
			\draw (2,2) -- (3,3);
			\draw[dotted] (0,0) -- (1,1) -- (1,3);
			\draw[dotted] (1,1) -- (3,1);
		\end{scope}
	\end{tikzpicture}
	\caption{The action of the generator $b_3$ of $P_3$ on the first two layers of $T^{(3)}$.}\label{fig:generator}
\end{figure}

We fix two generating sets for $K_r$,
\[
	\mathbb E_r = E_r \cup \{ b_r \} \quad\text{ and }\quad
	\mathbb S_r = A_r \cup b_r^{A_r},
\]
and establish some basic properties of the groups $K_r$.

\begin{lemma}\label{lem:infinite}
	Let $r \in \N_+$ be a positive integer. The group $K_r$ is self-similar, fractal and spherically transitive. In particular, it is infinite.
\end{lemma}

\begin{proof}
	The rooted group $A_r$ acts transitively on the first layer. Since rooted elements have trivial sections, self-similarity follows from the fact that all sections of $b_r$ are in $\mathbb E_r \subset K_r$. In fact, all elements of $\mathbb E_r$ appear as sections of $b_r \in \St_{K_r}(1)$. Conjugating by rooted elements, we may achieve any section if $b_r$ at any first layer vertex, thus $K_r$ is fractal. By the transitivity of $A_r$ the group $K_r$ acts transitively on the second layer, and inductively, $K_r$ is spherically transitive.
\end{proof}

Now we come to the core of our argument for establishing slow period growth. We prove an inequality between the length of an element and its sections at vertices of the second layer, using that the automorphism $b_r$ has short sections with respect to $\mathbb E_r$, but the only conjugates in $b_r^A$ aside from $b_r$ which have non-trivial section at the vertex $1_{A_r}$ are big with respect to $\mathbb E_r$. In preparation for the proof of \cref{Theorem}{thm:main}, we prove this inequality for a more general class of groups than just those of the form $K_r$. Therefore we need the following technical definition. Let $r \in \N_+$, and let $\tilde T$ be a rooted tree of type $(X_n)_{n \in \N_+}$ such that $X_1 = X_2 = A_r$. An element $b \in \St_{\Aut(\tilde T)}(1)$ is said to \emph{two-layer resemble} $b_r$ if the following three conditions hold:
\begin{enumerate}
	\item $b|_x = b_r|_x$ for $x \in A_r\setminus\{1_{A_r}\}$,
	\item $b|_{1_{A_r}} \in \St(1)$,
	\item $b|_{1_{A_r}x} = b_r|_{1_{A_r}x}$ for $x \in A_r\setminus\{1_{A_r}\}$.
\end{enumerate}
A group $\mathcal G \leq \Aut(\tilde T)$ is said to \emph{two-layer resemble $K_r$ with respect to $b$} if it is generated by a set $\mathcal E = E_r \cup \langle b \rangle$, where $b$ is an automorphism that two-layer resembles $b_r$.

Clearly $b_r$ two-layer resembles itself. Notice that the coset $b_r \St(2)$ contains many elements that do not two-layer resemble $b_r$, since the first (and second) layer sections of an element in $\St(2)$ do not need to be rooted. In fact, if the trees $\tilde T$ and $T^{(r)}$ coincide, the set of elements that two-layer resembles $b_r$ is equal to the coset $b_r \cdot \rst_{\Aut(T)}(1_{A_r}1_{A_r})$.

\begin{lemma}
	\label{lem:main technical}
	Let $\mathcal G \leq \Aut(\tilde T)$ be a group that two-layer resembles $K_r$ with respect to $b \in \Aut(\tilde T)$. Write $\mathcal S = A_r \cup \langle b \rangle^{A_r}$ and $\mathcal S'' = A_r \cup \langle b|_{1_{A_r}1_{A_r}} \rangle^{A_r}$. Then for all $g \in \mathcal G$ and $u \in \mathcal{L}_{\tilde T}(2)$ we have
	\[
		||g|_u||_{\mathcal S''} \leq \left\lceil {\,||g||_{\mathcal S}}/{r} \right\rceil.
	\]
\end{lemma}

\begin{proof}
	The reader less interested in the technicalities may consider this proof in its application to the example $b = b_r$, reading $\mathcal G = K_r$, $\mathcal E = \mathcal E' = \mathbb E_r$ and $\mathcal S = \mathcal S'' = \mathbb S_r$, avoiding some of the cumbersome notation necessary to deal with the more delicate construction proving \cref{Theorem}{thm:main}.
	
	It is sufficient to prove $||g|_u||_{\mathcal S''} \leq 1$ for all $g \in B_{\mathcal G}^{\mathcal S}(r)$. From this one derives the desired inequality by splitting a minimal $\mathcal S$-word representing $g$ into pieces of length at most $r$.
	
	We may write
	\[
		g = (b^{n_1})^{a_1} \dots (b^{n_{k-1}})^{a_{k-1}} a_k
	\]
	for some $a_i \in A_r$ and $n_i \in \Z$ with $i \in [1, k]$. For any $x \in A_r$ we have
	\[
		g|_x = (b^{n_1})^{a_1}|_x \dots (b^{n_{k-1}})^{a_{k-1}}|_x,
	\]
	hence the section at $x$ equals product of $k$ elements of the form $(b^{n})^{a}|_x = b^n|_{xa} = (b|_{xa})^n$ with $n \in \Z, a \in A_r$. If $x \neq a$, this section is equal to $(b_r|_{xa})^n$, by the first property of automorphisms two-layer resembling $b_r$. Otherwise we obtain $(b|_{1_{A_r}})^n$. Thus writing $\mathcal E' = E_r \cup \langle b|_{1_{A_r}} \rangle$ we see that $||g|_x||_{\mathcal E'} \leq k \leq r$.
	
	Now we look at $g|_{xy}$ for $y \in A_r$. Write $\underline{b}$ for $b|_{1_{A_r}}$. There are elements $a_{i, j} \in E_r$ and $n_i \in \Z$ with $i \in [1, k]$ and $j \in [1, m_i]$ for some $m_i \leq r$ such that
	\[
		g|_{x} = a_{1,1} \dots a_{1, m_1} \underline{b}^{n_1} a_{2,1} \dots a_{2, m_2} \underline{b}^{n_2} \dots a_{k-1, 1} \dots a_{k-1, m_{k_1}}\underline{b}^{n_{k-1}} a_{k,1} \dots a_{k, m_k}
	\]
	and $k-1 + \sum_{i = 1}^k m_i \leq r$, since every element on the right side is obtained as a section of an expression $(b|_{xa})^n$.
	
	Writing $\widehat{a_i} = a_{i,1} \dots a_{i, m_i}$ for $i \in [1, k]$ we obtain
	\begin{equation*}\label{eq:syllables h}
		g|_{x} = (\underline{b}^{n_1})^{\widehat{a_1}} (\underline{b}^{n_2})^{\widehat{a_1} \widehat{a_2}} \dots (\underline{b}^{n_{k-1}})^{\widehat{a_1} \dots \widehat{a_{k-1}}} \widehat{a_1}\dots\widehat{a_k}.\tag{$\ast$}
	\end{equation*}
	Using the second and the third property of automorphisms two-layer resembling $b_r$ we find, for all $n \in \Z$ and $a \in A_r\setminus\{y\}$,
	\[
		(\underline{b}^n)^{a}|_{y} = \underline{b}^n|_{ya} = (\underline{b}|_{1_{A_r}(ya)})^n = (b_r|_{ya})^n,
	\]
	hence the only generators of form $(\underline{b}^n)^a$ with non-trivial section at $1_{A_r}$ are powers of either $\underline{b}^y$ or $\underline{b}^{y\overline{e_i}}$ for some $i \in [0, r)$. Thus, calculating $g|_{xy}$, we might ignore all others. Assume that (\ref{eq:syllables h}) contains no generator of type $\underline{b}^y$. Then $g|_{xy}$ is a product of elements in $A_r$, hence of $\mathcal S''$-length $1$. Similarly, if it does not contain a generator $\underline{b}^{y\overline{e_i}}$, it is a power of $b|_{1_{A_r}1_{A_r}}$ and also of length $1$. Consequently, we have to exclude the case that both types appear in (\ref{eq:syllables h}). Assume for contradiction that this is the case. Without loss of generality we may suppose that $\underline{b}^{y}$ appears first. Then there exist $\ell_0, \ell_1 \in [1,k]$ such that
	\[
		\widehat{a_1} \dots \widehat{a_{\ell_0}} = y \quad\text{and}\quad \widehat{a_1} \dots \widehat{a_{\ell_1}} = y\overline{e_i}.
	\]
	Thus, in (\ref{eq:syllables h}), left of the $\ell_0$\textsuperscript{th} $\underline{b}$-symbol (which is associated to the generator $\underline{b}^y$) there appear at least $||y||_{\mathcal E'}$ letters from $\mathcal E'$. Between the $\ell_0$\textsuperscript{th} and the $\ell_1$\textsuperscript{th} $\underline{b}$-symbols appear at least $||y\overline{e_i}||_{\mathcal E'} \geq r - 1 - ||y||_{\mathcal E'}$ letters. Thus, also counting the at least two $\underline{b}$-symbols, we obtain
	\[
		r \geq ||g|_x||_{\mathcal E'} \geq ||y||_{E_r} + ||y\overline{e_i}||_{E_r} + 2 \geq r+1,
	\]
	a contradiction.
\end{proof}

Applying the lemma to $b = b_r$ and $G = K_r$ (and using the self-similarity of $K_r$) we obtain the following inequality.

\begin{lemma}\label{lem:reduction}
	Let $g \in K_r$ be an element and let $u \in \mathcal{L}_{T^{(r)}}(2)$. Then
	\[
		||g|_{u}||_{\mathbb S_r} \leq \left\lceil \,{||g||_{\mathbb S_r}}/{r} \right\rceil.
	\]
\end{lemma}

\begin{proof}[Proof of \cref{Theorem}{thm:secondary}]
	Notice that \cref{Proposition}{prop:periodicity}, \cref{Lemma}{lem:infinite} and \cref{Lemma}{lem:reduction} show that $K_r$ is an infinite $2$-group in case $r > 4$. 

	We prove $\pi_{K_r}^{\mathbb S_r}(n) \leq n^{1/(\log_{4}(r)- 1)}$ for every $n \in \N$ and $r > 4$. Clearly, choosing some big integer $r$, this proves the theorem.

	Let $g \in K_r$ be an element. Write $n = ||g||_{\mathbb{S}_r}$. Since $A_r$ is a group of exponent two, $g^2 \in \St_{K_r}(1)$ and $g^4 \in \St_{K_r}(2)$. 
	Consequently, the order of $g^4$ is the least common multiple of the orders of $g^4|_{u}$ for $u \in \mathcal L_{T^{(r)}}(2)$, which equals, since $K_r$ is a $2$-group, the maximum of their orders, i.e.\
	\[
		\ord(g) \leq 4 \cdot \max \{ \ord(g^4|_{u}) \mid u \in \mathcal{L}_{T^{(r)}}(2) \}.
	\]
	In view of \cref{Lemma}{lem:reduction}, we see $||g^4|_{u}||_{\mathbb S_r} \leq \lceil \frac {4 n} {r}\rceil$, so for $n \geq r$
	\[
		\pi_{K_r}^{\mathbb S_r}(n) \leq 4 \cdot \pi_{K_r}^{\mathbb S_r}\left(\lceil {4n}/{r}\rceil\right),
	\]
	hence, using that $K_r$ is generated by involutions,
	\[
		\pi_{K_r}^{\mathbb S_r}((\tfrac{r}{4})^k) \leq 4^k \pi_{K_r}^{\mathbb S_r}(1) = 2 \cdot 4^k.
	\]
	This implies
	\[
		\pi_{K_r}^{\mathbb S_r} \precsim \exp_4 \circ \log_{\frac {r}4} \sim n^{1/(\log_{4}(r)-1)}.\qedhere
	\]
\end{proof}


\section{Growing valency and the proof of \cref{Theorem}{thm:main}} 
\label{sec:proof_of_cref_theorem_thm_main}

We now construct a group $G$ with the properties described in \cref{Theorem}{thm:main}. To achive this we take the generators $b_r$ of the groups $K_r$ constructed in the previous chapter and build a single automorphism $d$ acting on a rooted tree with unbounded valency, that resembles some $b_{r_0}$ for two layers (where the valency is $2^{r_0}+1$), then uses one layer to increase the valency to $2^{r_1}+1$ for some $r_1 > r_0$, then resembles $b_{r_1}$ for two layers~\&c. This will allow us to use the reduction formulas for the $b_r$, but with (rapidly) increasing $r$.

The slowest period growth (using this construction) will be achived if one arranges the sequence $(r_n)_{n \in \N}$ to grow as fast as possible. For this there is a natural upper bound. We want the sections of $d$ at a given layer of valency $r_{n+1}+1$ to generate an elementary abelian $2$-group acting on the layer below, but can use no more than $2^{r_{n}}-1$ sections as generators. Hence the maximum possible increase of valency is given by the following function $f: \N \to \N$. Let $f(0) = 3$ and $f(k+1) = 2^{f(k)}-1$ for $k \in \N$. Since we aim to increase the valency of our tree on every third layer, we also introduce $f_3(k) = f(\lfloor k/3\rfloor)$, a function that takes every value of $f$ thrice. These functions grow very quickly.

\begin{lemma}\label{lem:growth of f}
	For all $k \in \N$ we have
	\(
		f(k) \geq \tetr_{2}(k).
	\)
\end{lemma}

\begin{proof}
	We use induction on $k$ for the statement $f(k) - 1 \geq \tetr_{2}(k)$. Clearly $f(0) - 1 = 2 \geq 1 = \tetr_{2}(0)$. Now for all $k > 0$
	\[
		f(k + 1) - 1 = 2^{f(k)}-2 \geq 2^{f(k)-1} \geq 2^{\tetr_{2}(k)} \geq \tetr_{2}(k+1).\qedhere
	\]
\end{proof}

Recall from the previous chapter that $A_r$ denotes a copy of the elementary abelian $2$-group with an (ordered) basis $E_r = \{e_0, \dots, e_{r-1}\}$. We now fix some enumeration (which may depend on $r$) $\{a_i \mid i \in [0, 2^r)\} = A_r$ for these groups, such that $a_0$ is the trivial element. Also recall also the translation map $a \mapsto \overline{a}^{(r)} = a\prod_{i = 0}^{r-1}e_i$ defined in the previous chapter. We introduce the superscript to make precise within which group we are translating.

Now we define $T$ as the rooted tree of type $(A_{f_3(k)})_{k \in \N}$. For any $k \equiv_3 0$ excluding $k = 0$, the $k$\textsuperscript{th}, $(k+1)$\textsuperscript{st} and $(k+2)$\textsuperscript{nd} layers of $T$ have valency $2^{f_3(k)}+1$. Write $T_k$ for the (isomorphism class) of any subtree of $T_u$ for some $u \in \mathcal{L}_T(k)$, i.e. $T_0 = T$ and $T_k$ of type $(A_{f_3(l)})_{l \geq k}$.

Again we view the group $A_{f_3(k)}$ as rooted automorphisms by their right multiplication action. Define a sequence of automorphisms $d_n \in \Aut(T_k)$ for $k \in \N$ by
\begin{align*}
	d_k &= (1_{A_{f_3(k)}}: d_{k+1};\, \overline{e_i}^{(f_3(k))}: e_i;\, \ast: \id) &\text{ for } k \equiv_3 0, 1 \text{ and}\\
	d_k &= (1_{A_{f_3(k)}}: d_{k+1};\, a_i: e_i \text{ for } i \in [1, 2^{f_3(k)}) ) &\text{ for } k \equiv_3 2.
\end{align*}
Finally, we define $G_k = \langle A_{f_3(k)} \cup \{d_k\} \rangle \leq \Aut(T_k)$, and write $G$ for $G_0$.

Note that among the sections of $d_k$ are all the elements of $E_{k+1}$. Using this, we see that, for every $v \in T$ of length $k$, we have $G|_v = G_k$, and $G$ acts spherically transitively on $T$.

For $k \in \N$, define $\mathrm S_k = A_{f_3(k)} \cup \{d_k\}^{A_{f_3(k)}}$ and $\mathrm E_k = E_{f_3(k)} \cup \{d_k\}$, filling the rôles of $\mathbb S_r$ and $\mathbb E_r$ of \cref{Section}{sec:proof_of_the_theorem}. Both are generating sets for $G_k$. Note that $d_k^2 = 1$, hence both sets consist of involutions.

\begin{lemma}\label{lem:reduction step}
	Let $k \in \N$ be a positive integer such that $k \equiv_3 0$ and $g \in G_k$ an element. Then for all $v \in \mathcal{L}_{T_k}(2)$ we have
	\[
		||g|_v||_{\mathrm S_{k+2}} \leq \left\lceil \frac{||g||_{\mathrm S_n}}{ f(k/3)}\right\rceil.
	\]
\end{lemma}

\begin{proof}
	We apply \cref{Lemma}{lem:main technical}. This is possible since by definition $d_k$ two-layer resembles $b_{f(k/3)}$. Notice that $\mathcal S = \mathrm S_k$ and $\mathcal S'' = \mathrm S_{k+2}$.
\end{proof}

\begin{lemma}\label{lem:change of valency}
	Let $k \in \N$ and let $g \in G_k$. Then for all $x \in \mathcal{L}_{T_k}(1)$
	\[
		||g^2|_x||_{\mathrm S_{k+1}} \leq ||g||_{\mathrm S_{k}} + 1.
	\]
\end{lemma}

\begin{proof}
	Since $\langle d_k\rangle^{A_{f_3(k)}}$ is closed under conjugation with $A_{f_3(k)}$, we may write
	\(
		g = d_k^{a_1} \dots d_k^{a_{l}}
	\)
	for $l = ||g||_{\mathbb S_n}$, for some $a_i \in A_{f_3(k)}$ for $i \in [1, l]$. Then
	\(
		g|_x = d_k^{a_1}|_x \dots d_k^{a_{l}}|_x.
	\)
	Now at most every second expression $d_k^{a_i}|_x$ can evaluate to $d_k$. Otherwise there is some $i$ such that $a_i = a_{i+1} = x$, which implies
	\[
		g = d_k^{a_1} \dots d_k^{a_{i-1}} d_k^u d_k^u d_k^{a_{i+2}} \dots d_k^{a_{l}} = d_k^{a_1} \dots d_k^{a_{i-1}} d_k^{a_{i+2}} \dots d_k^{a_{l}}.
	\]
	But then $||g||_{\mathrm S_{n}} \leq l-2$, a contradiction. Hence there are at most $\lceil l/2 \rceil$ symbols $d_k$ in the product $d_k^{a_1}|_x \dots d_k^{a_{l}}|_x$, and we have $||g|_x||_{\mathrm S_k} \leq \lceil \frac 1 2 ||g||_{\mathrm S_{k}} \rceil$. Now
	\[
		||g^2|_x||_{\mathrm S_{k+1}} = ||g|_x g|_{x.g}||_{\mathrm S_{k+1}} \leq ||g|_x||_{\mathrm S_{k+1}} ||g|_{x.g}||_{\mathrm S_{k+1}} \leq 2 \lceil \tfrac 12||g||_{\mathrm S_k} \rceil \leq ||g||_{\mathrm S_{k}} + 1.\qedhere
	\]
\end{proof}

\begin{lemma}\label{lem:reduction full}
	Let $k \equiv_3 0$ and let $g \in G_k$. Then for all $u \in \mathcal{L}_{T_k}(3)$
	\[
		||g^8|_u||_{\mathrm S_{k+3}} \leq \left\lceil \frac{4 \cdot ||g||_{\mathrm S_{k}}}{ f(k/3)} \right\rceil + 1.
	\] 
\end{lemma}

\begin{proof}\belowdisplayskip=-12pt
	Since $A_{f_3(k)}$ and $A_{f_3(k+1)}$ are of exponent two, we have $g^4 \in \St_{G_k}(2)$. Hence $g^8|_u = (g^4|_{u_1u_2})^2|_{u_3}$, where $u = u_1u_2u_3$. Now
	\begin{align*}
		||g^8|_u||_{\mathrm S_{k+3}} &= ||(g^4|_{u_1u_2})^2|_{u_3}||_{\mathrm S_{k+3}}&&\\
		&\leq ||g^4|_{u_1u_2}||_{\mathrm S_{k+2}} + 1&& \text{(by \cref{Lemma}{lem:change of valency})}\\
		&\leq \left\lceil\frac{||g^4||_{\mathrm S_{k}}}{f(k/3)} \right\rceil + 1 && \text{(by \cref{Lemma}{lem:reduction step})}\\
		&\leq \left\lceil\frac{4 \cdot ||g||_{\mathrm S_{k}}}{f(k/3)} \right\rceil + 1.
	\end{align*}\qedhere
\end{proof}

\begin{lemma}\label{lem:G periodic}
	The group $G$ is a $2$-group.
\end{lemma}

\begin{proof}
	This follows from \cref{Proposition}{prop:periodicity} and \cref{Lemma}{lem:reduction step}. Using the notation of \cref{Proposition}{prop:periodicity}, let $n = 10$. Since $G|_u/\St_{G|_u}(1)$ is an elementary abelian $2$-group for all $u \in T_0$, we see that $\exp(G|_u/\St_{G|_u}(n)) \leq 2^n$. Now, irregardless of the value of $k$ modulo $3$, taking the $10$\textsuperscript{th} section of some $g \in G_k$ allows us to invoke \cref{Lemma}{lem:reduction step} at least three times. Hence for all $w \in \mathcal{L}_{T_k}(10)$
	\[
		||g|_w||_{\mathrm S_{k+10}} \leq \frac{||g||_{\mathrm S_k}}{f(0)f(1)f(2)} = \frac{||g||_{\mathrm S_k}}{3\cdot 7 \cdot 127} < \frac{||g||_{\mathrm S_k}}{2^{10}}
	\]
	and we conclude that $G$ is a $2$-group.
\end{proof}

\begin{proof}[Proof of \cref{Theorem}{thm:main}]
	Let $n, k \in \N$ with $k \equiv_3 0$, and let $g \in B_{G_k}^{\mathrm S_k}(n)$. Since $\exp(A_l) = 2$ for all $l \in \N$, the $2^3$-power of $g$ fixes the third layer of $T_n$, hence
	\[
		\ord(g^8) \leq 8 \cdot \max\{\ord(g^8|_v) \mid v \in \mathcal{L}_{T_k}(3)\}.
	\]
	Now \cref{Lemma}{lem:reduction full} implies
	\[
		\pi_{G_k}^{\mathrm S_k}(n) \leq 8 \cdot \pi_{G_{k+3}}^{\mathrm S_{k+3}}\left(\left\lceil \tfrac{4 \cdot n}{ f(k/3)}\right\rceil + 1\right).
	\]
	Writing $v_k(n) = \lceil 4 \cdot n/f(k/3) \rceil + 1$ and
	\[
		u(n) = \min \{ l \in \N \mid v_l(v_{l-1}(\dots (v_0(n)) \dots)) = 2 \}
	\]
	we find
	\[
		\pi_{G}^{\mathrm S}(u(n)) \leq 8^{n} \cdot \pi_{G_{3n}}^{\mathrm S_{3n}}(2).
	\]
	Now, using the same argument as before, we see that $\pi_{G_{3n}}^{\mathrm S_{3n}}(2) \leq 4$ by \cref{Lemma}{lem:reduction step}. Thus, deriving $\tetr_2 \precsim u(n)$ from \cref{Lemma}{lem:growth of f}, we obtain
	\[
		\pi_{G}\precsim \exp_8 \circ \slog_{2}.\qedhere
	\]
\end{proof}


\section{Lawlessness growth} 
\label{sec:lawlessness}

Let $G$ be a lawless group generated by a finite set $S$. By the definition of lawlessness, the image of the word map $w(G^m)$ is non-trivial for every reduced word $w \in F_m\setminus\{1\}$ in $m$ letters, $m \in \N$. We may define the \emph{complexity of $w$ in $G$ with respect to $S$} by
\[
	\chi_G^S(w) = \min \left\{ \; \sum_{i = 1}^m ||g_i||_S \mid \underline g = (g_i)_{i = 1}^m \in G^m, w(\underline g) \neq 1 \right\} \in \N.
\]
Now the \emph{lawlessness growth function} $\mathcal{A}_G^S: \N \to \N$ of $G$ with respect to $S$ is defined by 
\[
	\mathcal{A}_G^S(n) = \max \{ \chi_G^S(w) \mid w \in F_m\setminus\{1\} \text{ with } ||w|| \leq n \}.
\]
This definition is due to Bradford, first given in \cite{Bra}, where he proves the independence of the growth type from the choice of generating set and establishes a connexion to the period growth in the case of periodic $p$-groups.

\begin{prop}{\cite{Bra}}\label{prop:torsion to lawless}
	Let $G$ be a finitely generated lawless periodic $p$-group for some prime $p$ and $f: \N \to \N$ some function. Then
	\(
		\pi_G^S(n) \leq f(n)
		\text{ implies }
		\mathcal A_G^S(f(n)) \geq n.
	\)
\end{prop}

Using this, we give examples of groups with large lawlessness growth (cf.\ \cite[Question 10.2]{Bra}) by proving that the groups constructed in the previous sections are in fact lawless. As a consequence of \cref{Theorem}{thm:main} and \cref{Proposition}{prop:torsion to lawless} we obtain the following corollary.

\begin{cor}
	There is a finitely generated lawless group $G$ such that
	\[
		\mathcal A_G^S \gtrsim \tetr_{2} \circ \log_8.
	\]
\end{cor}

It remains to prove that the group $G$ of \cref{Theorem}{thm:main} is lawless. We prove that it is weakly branch, which is sufficient by \cite{Abe05}. Our proof is technical, but also establishes that the groups $K_r$ are weakly branch for all integers $r > 5$. To avoid some obstacles appearing for small valencies, we look at $G_{6}$ instead of $G = G_0$, for which the proof of \cref{Theorem}{thm:main} works verbatim, except for the number of generators. Thus in the remainder of this section, we write $G$ for $G_6$ and define the function $f$ prescribing the valencies of the tree upon which $G$ acts by $f(0) = 127$ and $f(n+1) = 2^{f(n)}-1$ for $n > 0$.

\begin{lemma}\label{lem:resemblence reprod}
	Let $r \in \N_{> 5}$ and let $\mathcal G \leq \Aut(\widetilde T)$ be a group that two-layer resembles $K_r$ with respect to $b$. Define
	\begin{align*}
		N &= \langle [b, e_i, e_j] \mid i, j \in [0, r), i \neq j \rangle^{\mathcal G} \leq \Aut(\widetilde T), \quad \text{and}\\
		\underline N &= \langle [b|_{1_{A_r}}, e_i, e_j] \mid i, j \in [0, r), i \neq j \rangle^{\mathcal G|_{1_{A_r}}} \leq \Aut(\widetilde T|_{1_{A_r}})
	\end{align*}
	Then for every $x \in \mathcal L_{\widetilde T}(1)$ we have
	\(
		\rst_{N}(x) \geq \underline{N}.
	\)
\end{lemma}

\begin{proof}
	Write $c_{i,j} = [b, e_i, e_j]$ for the (normal) generators of $N$. Clearly $N \leq \St_{\mathcal G}(1)$. We compute
	\begin{align*}
		c_{i,j}|_x &=
		\begin{cases}
			b|_{1_{A_r}} &\text{ if }x \in \{1_{A_r}, e_i, e_j, e_ie_j\},\\
			e_t &\text{ if }x \in \{\overline{e_t}, \overline{e_t e_i}, \overline{e_t e_j}, \overline{e_t e_i e_j}\} \text{ and } t \in [0,r)\setminus\{i,j\},\\
			e_ie_j &\text{ if }x \in \{\overline{1_{A_r}}, \overline{e_i e_j}, \overline{e_i}, \overline{e_j}\},\\
			\id &\text{ otherwise}.
		\end{cases}
	\end{align*}
	Let $i, j, k, m, n$ be pairwise distinct elements of $[0, r)$ (here we need $r > 4$). We look at $[c_{i,j}, c_{m,n}^{\overline{e_k}}]$. Since both $c_{i,j}$ and $c_{m,n}^{\overline{e_k}}$ are in $\St(1)$, taking the commutator commutes with taking sections. All sections except $b|_{1_{A_r}}$ commute, so we have $[c_{i,j}, c_{m,n}^{\overline{e_k}}]|_x = \id$ for all $x \not\in \{1_{A_r}, e_i, e_j, e_ie_j, \overline{e_k}, \overline{e_ke_m}, \overline{e_ke_n}, \overline{e_ke_me_n}\}$.
	Since $r > 5$, all these vertices are distinct. Furthermore, for the remaining cases we calculate
	\begin{align*}
		[c_{i,j}, c_{m,n}^{\overline{e_k}}]|_x =
		\begin{cases}
			[b|_{1_{A_r}},e_k] &\text{ if }x = 1_{A_r},\\
			[e_k, b|_{1_{A_r}}] &\text{ if }x = \overline{e_k},\\
			[b|_{1_{A_r}}, \id] = \id &\text{ if }x \in \{e_i,e_j,e_ie_j\},\\
			[\id, b|_{1_{A_r}}] = \id &\text{ if }x \in \{\overline{e_ke_m},\overline{e_ke_n},\overline{e_ke_me_n}\}.
		\end{cases}
	\end{align*}
	Now let $l \in [0,r)\setminus\{i, j, k\}$. Then $c_{i,j}^{\overline{e_l}}|_{1_{A_r}} = e_l$ and $c_{i,j}^{\overline{e_l}}|_{\overline{e_k}} = c_{i,j}|_{e_ke_l} = \id$.
	Consequently
	\[
		[c_{i,j}, c_{m,n}^{\overline{e_k}},c_{i,j}^{\overline{e_l}}]|_x = \begin{cases}
			[b|_{1_{A_r}}, e_k, e_l] &\text{ if }x = 1_{A_r},\\
			\id &\text{ else},
		\end{cases}
	\]
	thus $\rst_N(1_{A_r}) \geq \langle [b|_{1_{A_r}}, e_i, e_j] \mid i, j \in [0, r), i \neq j \rangle$. Since $\{b^{\overline{e_i}}|_{1_{A_r}} \mid i \in [0, r) \} \cup \{ b|_{1_{A_r}} \}$ generates $\mathcal G|_{1_{A_r}}$, for every $g \in \mathcal G|_{1_{A_r}}$ we find an element $\widehat{g} \in \St_{\mathcal G}(1)$ such that $\widehat g|_{1_{A_r}} = g$. Conjugating with these elements, we find $\rst_N(1_{A_r}) \geq \underline{N}$. Since $\mathcal G$ acts transitively on the first layer, all rigid vertex stabilisers are conjugate, and we obtain the result.
\end{proof}

\begin{prop}
	Let $r \in \N_{> 5}$. Then $K_r$ is weakly regular branch, hence lawless.
\end{prop}

\begin{proof}
	This follows directly from \cref{Lemma}{lem:resemblence reprod}, since the two normal subgroups $N, \underline{N}$ are equal in the case of $K_r$.
\end{proof}

\begin{lemma}\label{lem:semi fractal}
	Let $k \in \N$ and $x \in \mathcal L_{T_k}(1)$. Then $\St_{G_k}(1)|_x \geq {G_{k+1}}$.
\end{lemma}

\begin{proof}
	Observe $\mathrm E_{k+1} = \{ d_k|_x \mid x \in \mathcal L_{T_k}(1) \}$ and that $G_k$ acts transitively on $L_{T_k}(1)$.
\end{proof}

\begin{prop}
	The group $G$ is a weakly branch group, hence a lawless group.
\end{prop}

\begin{proof}
	Let $k \in \N$ be an integer such that $k \equiv_3 0$. We adopt the following notation to better distinguish between the generators of $A_{f_3(k)}$ and $A_{f_3(k+3)}$. If $a = e_{i_0} \dots e_{i_t}$ is a non-trival element of $A_{f_3(k)}$, we write $\underline{e}_{i_0\dots i_t}$ for the generator $d_{k+2}|_a$ of $A_{f_3(k+3)}$. Each element of $E_{f_3(k+3)}$ appears in this way. Define
	\begin{align*}
		N_k &= \langle [d_k, e_i, e_j] \mid i, j \in [0, f_3(k)), i \neq j \rangle^{G_k},\quad \text{ and}\\
		  M_k &= \left\langle [[d_k, a_1], [d_k, a_2]^g] \;\middle|\;
		  \begin{aligned}
		  & g \in G_k, a_1 = \underline{e}_j\underline{e}_{ij}\underline{e}_{l}\underline{e}_{il}, a_2 = \underline{e}_n\underline{e}_{mn}\underline{e}_{s}\underline{e}_{ms},\\
		  & i,j,l,m,n,s \in [0, f_3(k-1)) \text{ pairwise distinct }
		  \end{aligned}
		\right\rangle^{G_k}.
	\end{align*}
	The group $G_k$ two-layer resembles $P_{f_3(k)}$, thus \cref{Lemma}{lem:resemblence reprod} implies $\rst_{N_{k+1}}(u) \geq N_{k+2}$ for $u \in \mathcal L_{T_{k+1}}(1)$. We show that
	\begin{align*}
		\rst_{M_{k}}(w) &\geq N_{k+1} \text{ for $k > 0$, and }\tag{$\dagger$}\label{eq:first step}\\
		\rst_{N_k+2}(v) &\geq M_{k+3}.\tag{$\ddagger$}\label{eq:second step}
	\end{align*}
	Using this, we see that for all $u \in \mathcal L_T(l)$
	\[
		\rst_G(u) \geq \begin{cases}
			M_l &\text{ if }l \equiv_3 0,\\
			N_l &\text{ otherwise}.
		\end{cases}
	\]
	Since $N_l$ and $M_l$ are non-trivial for all $l \in \N$, this shows that $G$ is a weakly branch group.
	
	In both cases it is enough to show that the normal generators of $N_{k+1}$, resp.\ $M_{k+3}$, are contained in the rigid vertex stabiliser of $1_{A_{f_3(k+1)}}$, resp.\ $1_{A_{f_3(k+3)}}$. Using \cref{Lemma}{lem:semi fractal}, we find the full normal subgroup within the rigid vertex stabiliser of $1_{A_{f_3(k)}}$, and since $G_k$ acts spherically transitive, all rigid vertex stabilisers of the same layer are conjugate.
	
	We first prove (\ref{eq:first step}). Let $k > 0$.
	Let $a_1, a_2 \in B_{A_{f_3(k)}}^{E_{f_3(k)}}(4)$ such that $[[d_k,a_1],[d_k,a_2]]$ is a normal generator of $M_k$. Calculate
	\[
		[d_{k}, a_1]|_x = d_{k}d_{k}^{a_1}|_x = \begin{cases}
			d_{k+1} &\text{ if }x \in \{1_{A_{f_3(k)}}, a_1\},\\
			e_t &\text{ if }x \in \{\overline{e_t}, \overline{e_t}a_1\}, \text{ for some } t \in [0, f_3(k)),\\
			\id &\text{ otherwise.}
		\end{cases}
	\]	
	We want to compute $[[d_{k}, a_1], [d_{k}, a_2]^{\overline{e}_s}]$ for arbitrary $s \in [0, f_3(k))$. The set of vertices where this element might have non-trivial sections is $\{1_{A_{f_3(k)}}, a_1, \overline{e_s}, \overline{e_s}a_2\}$.
	
	We now prove that the sections $[d_{k}, a_1]|_{\overline{e_s}a_2}$ and $[d_{k}, a_2]^{\overline{e_s}}|_{a_1}$ are trivial, i.e.\ that
	\begin{align*}
		\overline{e_s}a_2 &\notin \{ 1_{A_{f_3(k)}}, a_1, \overline{e_t}, \overline{e_t}a_1 \mid t \in [0, f_3(k)), \quad\text{and}\\
		\overline{e_s}a_1 &\notin \{ 1_{A_{f_3(k)}}, a_2, \overline{e_t}, \overline{e_t}a_2 \mid t \in [0, f_3(k)).
	\end{align*}
	Now $||\overline{e_s}a_2||_{A_{f_3(k)}} \geq f_3(k) - 5$, hence $\overline{e_s}a_2$ is neither trivial nor equal to $a_1$ of length $4$. Here we use that $f_3(k) \geq f(0) > 9$. Finally
	\(
		\overline{e_t}a_1 = \overline{e_s}a_2
	\)
	implies $a_1e_s = a_2e_t$, which contradicts the definition of $a_1$ and $a_2$. This proves the first, and by analogy the second, non-inclusion statement above.
		
	Thus we find
	\[
		[[d_{k}, a_1], [d_{k}, a_2]^{\overline{e}_s}]|_x = \begin{cases}
			[d_{k+1}, e_s] &\text{ if }x = 1_{A_{f_3(k)}},\\
			[e_s, d_{k+1}] &\text{ if }x = \overline{e_s},\\
			\id &\text{ otherwise.}
		\end{cases}
	\]
	For every $q \in [0,f_3(k))\setminus\{s\}$ we obtain $$h = [[d_{k}, a_1], [d_{k}, a_2]^{\overline{e}_s}, [d_{k}, a_1]^{\overline{e}_q}] \in \rst_{M_{k}}(1_{A_{f_3(k)}}),$$
	such that $h|_{1_{A_{f_3(k)}}} = [d_{k+1}, e_s, e_q]$. This concludes the proof of (\ref{eq:first step}).
	
	We now prove (\ref{eq:second step}). Write $c_{i,j}$ for the element $[d_{k+2}, e_i, e_j] \in N_{k+2}$, where $i, j \in [0, f_3(k+2))$ are two distinct integers. Observe that
	\begin{align*}
		c_{i,j}|_{1_{A_{f_3(k+2)}}} = d_{k+3} \underline{e}_{i} \underline{e}_{j} \underline{e}_{ij}
	\end{align*}
	and that $c_{i,j}|_u \in A_{f_3(k+3)}$ for all $u \in \mathcal L_{T_{k+2}}(1)$ except the (distinct) vertices $1_{A_{f_3(k+2)}}$, $e_i$, $e_j$ and $e_ie_j$.
	Thus for $l \in [0, f_3(k+2))\setminus \{i, j\}$ we compute
	\[
		[c_{i,j}, c_{i,l}]|_x = \begin{cases}
			[d_{k+3} \underline{e}_{i} \underline{e}_{j} \underline{e}_{ij}, d_{k+3} \underline{e}_{i} \underline{e}_{l} \underline{e}_{il}] &\text{ if }x = 1_{A_{f_3(k+2)}},\\
			\text{possibly non-trivial} &\text{ if }x \in \{1_{A_{f_3(k)}}, e_i, e_j, e_l, e_ie_j, e_ie_l\},\\
			\id  &\text{ otherwise}.\\
		\end{cases}
	\]
	By \cref{Lemma}{lem:semi fractal} there is an element $\widehat g_0 \in \St_{G_{k+2}}(1)$ such that $\widehat g_0|_{1_{A_{f_3(k+2)}}} = \underline{e}_{i}\underline{e}_{j}\underline{e}_{ij}$. Now
	\[
		[c_{i,j}, c_{i,l}]^{\widehat g_0}|_{1_{A_{f_3(k+2)}}} = [d_{k+3} \underline{e}_{i} \underline{e}_{j} \underline{e}_{ij}, d_{k+3} \underline{e}_{i} \underline{e}_{l} \underline{e}_{il}]^{\underline{e}_{i}\underline{e}_{j}\underline{e}_{ij}} = [d_{k+3}, \underline{e}_{j}\underline{e}_{l}\underline{e}_{ij}\underline{e}_{il}],
	\]
	and the set of vertices $x$ such that $[c_{i,j}, c_{i,l}]^{g_0}|_x$ is possibly non-trivial is, as for $[c_{i,j}, c_{i,l}]$, the set $\{1_{A_{f_3(k)}}, e_i, e_j, e_l, e_ie_j, e_ie_l\}$.
	
	Let $g \in G_{k+3}$. There is an element $\widehat g_1 \in \St_{G_k}(1)$ such that $\widehat g_1|_{1_{A_{f_3(k)}}} = g$. We conclude that for three pairwise distinct integers $m, n, s \in [0, f_3(k+2))\setminus \{i, j, l\}$ (which is possible since the minimum value of $f_3$ greater then $5$)
	\begin{align*}
		[[c_{i,j}, c_{i,l}], [c_{m,n}, c_{m,s}]^{\widehat{g}}]|_{1_{A_{f_3(k)}}}
		&= [[d_{k+2}, \underline{e}_{j} \underline{e}_{ij} \underline{e}_{l} \underline{e}_{il}], [d_{k+2}, \underline{e}_{n} \underline{e}_{mn} \underline{e}_{s} \underline{e}_{ms}]^g],
	\end{align*}
	while all other sections are trivial, hence $\rst_{N_k}(1_A) \geq M_{k+1}$.
\end{proof}


\section{Open questions and related concepts} 
\label{sec:open_questions}

In \cite{BS01}, the authors refer to an unpublished text of Leonov \cite{Leo99}, where he establishs a connexion between the word growth and the period growth of the Grigorchuk group. It seems plausible that there is such a connexion: slow word growth makes for few elements of a given length, hence for a smaller set of candidates that might have big order. Consequently, we pose the following refinement of the question of Bradford.
\begin{description}
	\item[Q3] Is there an infinite finitely generated residually finite periodic group of exponential word growth and sublinear period growth?
\end{description}
To answer this, it would be sufficient to prove that the groups constructed in \cref{Theorem}{thm:main} and \cref{Theorem}{thm:secondary} are of exponential growth, but we doubt that this is true. In view of the numerical relation between the word and period growth in the Grigorchuk group, we think that the groups $G$ and $G_\epsilon$ are interesting candidates for groups of slow intermediate word growth. Thus we ask:
\begin{description}
	\item[Q4] Of what growth type is the word growth of $G$ and of $G_\epsilon$?
\end{description}


\end{document}